\documentclass{amsart}

\usepackage{amssymb,amsmath,amsthm,color}

\usepackage[bookmarksopen=true,bookmarksnumbered=true, bookmarksopenlevel=2, colorlinks=true,linkcolor=mblue,
citecolor=mblue,urlcolor=mblue, pdfstartview=FitH]{hyperref}

\definecolor{mblue}{rgb}{0,0,.8}

\newcommand{\N}{\mathbb N}
\newcommand{\Z}{\mathbb Z}
\newcommand{\Q}{\mathbb Q}
\newcommand{\Qbar}{\overline{\Q}}
\newcommand{\Zbar}{\overline{\Z}}
\newcommand{\F}{\mathbb F}

\newcommand{\R}{\mathbb R}

\newcommand{\p}{\mathfrak p}

\newcommand{\OO}{\mathcal O}

 \newcommand{\abcd}[4]{\left(
         \begin{smallmatrix}#1&#2\\#3&#4\end{smallmatrix}\right)}

\newcommand{\TT}{\mathbb{T}}

\newcommand{\fS}{\mathfrak{S}}
\newcommand{\fR}{\mathfrak{R}}

\newcommand{\rhobar}{\overline{\rho}}

\newcommand{\Zmod}[1]{\overline{\Z/({#1})}}

\newtheorem{thm}{Theorem}
\newtheorem{lem}[thm]{Lemma}
\newtheorem{dfn}[thm]{Definition}
\newtheorem{prop}[thm]{Proposition}
\newtheorem{cor}[thm]{Corollary}
\newtheorem{conj}{Conjecture}
\newtheorem{rem}{Remark}
\newtheorem{question}{Question}

\DeclareMathOperator{\GL}{GL} \DeclareMathOperator{\SL}{SL}  
\DeclareMathOperator{\Gal}{Gal}

\DeclareMathOperator{\Frob}{Frob}

\DeclareMathOperator{\Sym}{Sym}

\begin{document}

\title{On certain finiteness questions in the arithmetic of modular forms}
\author[Ian Kiming, Nadim Rustom, Gabor Wiese]{Ian Kiming, Nadim Rustom, Gabor Wiese}
\address[Ian Kiming, Nadim Rustom]{Department of Mathematical Sciences, University of Copenhagen, Universitetsparken 5, DK-2100 Copenhagen \O ,
Denmark.}
\email{kiming@math.ku.dk}
\email{rustom@math.ku.dk}
\address[Gabor Wiese]{Universit\'e du Luxembourg, Facult\'e des Sciences, de la Technologie et de la Communication, 6, rue Richard Coudenhove-Kalergi, L-1359 Luxembourg, Luxembourg}
\email{gabor.wiese@uni.lu}


\begin{abstract}
We investigate certain finiteness questions that arise naturally when studying approximations modulo prime powers of $p$-adic Galois representations coming from modular forms. We link these finiteness statements with a question by K.\ Buzzard concerning $p$-adic coefficient fields of Hecke eigenforms.

Specifically, we conjecture that for fixed $N$, $m$, and prime $p$ with $p$ not dividing $N$, there is only a finite number of reductions modulo $p^m$ of normalized eigenforms on $\Gamma_1(N)$.

We consider various variants of our basic finiteness conjecture, prove a weak version of it, and give some numerical evidence.
\end{abstract}

\maketitle

\section{Introduction}\label{intro} Let $p$ be a prime number. For theoretical and practical purposes, $p$-adic modular Galois representations are naturally
studied through their modulo~$p^m$ approximations. As explained below, with this article we would like to contribute to the ``weight aspect'' by presenting
theoretical results, numerical data and proposing some guiding finiteness conjectures.

Although our motivation stems foremostly from Galois representations, in this article we work with modular forms and restrict to the case of classical (elliptic) modular forms, as they already present many problems that are not yet well understood. Moreover, we chose to consider approximations modulo~$p^m$ for fixed~$m$, instead of varying or ``big enough''~$m$. This choice is forced upon us by the underlying motivation of our work. We explain this motivation below in \ref{subsection:motivation}.

\subsection{Definition of ``modulo~\texorpdfstring{$p^m$}{pm}''.} We first need to define the term ``modulo~$p^m$''. We fix once and for all algebraic closures $\Qbar$ (containing $\Zbar$, its ring of integers), $\Qbar_p$ (containing $\Zbar_p$, the elements integral over~$\Z_p$) of $\Q$ and $\Q_p$, respectively, as well as an embedding $\Qbar \hookrightarrow \Qbar_p$, which we will tacitly be using. Let $v_p$ be the normalized ($v_p(p)=1$) valuation on~$\Qbar_p$. The two natural requirements that (1) ``modulo~$p$'' should mean ``modulo~$\p$'' when coefficients/traces lie in some extension $K/\Q_p$ with $\p$ above~$p$, and that (2) the meaning of ``modulo~$p^m$'' be invariant under extensions of the field of coefficients/traces, force upon us the following definition, introduced in~\cite{XaGa} and utilized in~\cite{ckw}: For $m \in \N$ define
$$
\Zmod{p^m} = \Zbar_p/\{x \in \Zbar_p \;|\; v_p(x)>m-1\}.
$$
For $a,b \in \Zbar_p$ we say that $a \equiv b \mod p^m$ if the images of $a,b$ in $\Zmod{p^m}$ coincide. More concretely, for $a,b \in \OO_K \subset \Zbar_p$ (the valuation ring of~$K/\Q_p$) we have $a \equiv b \mod p^m$ if and only if $a-b \in \p^{e_{K/\Q}(m-1)+1}$, where $e_{K/\Q}$ is the ramification index. Moreover, by ``reducing modulo~$p^m$'' we understand taking the image under the natural maps $\Zbar_p \twoheadrightarrow \Zmod{p^m}$ or $\OO \twoheadrightarrow \OO/\p^{e_{K/\Q}(m-1)+1}$, respectively.

\subsection{Finiteness conditions and conjectures.}\label{subsection:finiteness_conditions} Before formulating the various finiteness conditions, our conjectures, and our results, we need to set up some basic notation.

For $N \in \N$, we let $\fS(N)$ be the set of normalized newforms on $\Gamma_1(N)$ of some weight $k$. Thus, every member of $\fS(N)$ is in particular a normalized cuspidal eigenform for all Hecke operators. Here, ``all Hecke operators'' obviously refers to level $N$. In the following we will also more generally be considering normalized cuspidal eigenforms for all Hecke operators, and these will be referred to simply as ``eigenforms''. Of course, ``eigenforms'' is a more general concept than newforms, as eigenforms might be oldforms.

The eigenvalue of any Hecke operator $T_n$ will always be denoted $a_n(f)$; it coincides with the $n$-th coefficient of the standard $q$-expansion. For $f\in \fS(N)$ we have an attached $p$-adic Galois representation $\rho_{f,p}: G_\Q \to \GL_2(\Zbar_p)$, where $G_\Q$ is shorthand for $\Gal(\Qbar/\Q)$. For $m \in \N$ we denote by $\rho_{f,p,m}$ the reduction of $\rho_{f,p}$ modulo~$p^m$, i.e., the composition of $\rho_{f,p}$ with the map induced from $\Zbar_p \twoheadrightarrow \Zmod{p^m}$. As usual, we write $\rhobar_{f,p} := \rho_{f,p,1}$. In general, these Galois representations depend on the choice of a lattice, but we will only be working with their characters, which are well-defined in all cases. See \ref{subsection:galois_representation} for more details.

\begin{dfn}\label{definition:basic_definitions} Fix $N\in\N$, $p$ a prime not dividing $N$, and $m\in\N$.

As above we denote by $\fS(N)$ the set of all normalized newforms on $\Gamma_1(N)$ and some weight~$k \ge 1$ (we allow $k=1$ in the definition, but in the article weight~$1$ will not play any role.) Elements $f\in\fS(N)$ are identified with their standard $q$-expansions.

Let $\fR(N)$ denote the set of all characters of all $p$-adic Galois representations arising from elements $f\in\fS(N)$.

We let $\fS_m(N)$ denote the set of reductions modulo $p^m$ of the standard $q$-ex\-pan\-sions of the elements of~$\fS(N)$.
Elements in $\fS_m(N)$ are called strong eigenforms modulo $p^m$.
Similarly, we let $\fR_m(N)$ be the set of reductions modulo $p^m$ of $\fR(N)$ (in the sense of composition with the natural
projection).

We work with the `naive' definition of modular forms with coefficients in a ring~$R$: the $R$-submodule of $R[[q]]$ spanned by the modular forms with integral standard $q$-expansions.
The space of cusp forms on $\Gamma_1(N)$ of weight $k$ and coefficients in $\Zmod{p^m}$ is denoted by $S_k(\Zmod{p^m})$. The Hecke operators $T_n$ act on the space $S_k(\Zmod{p^m})$. A weak eigenform modulo $p^m$ is a normalized eigenform for all $T_n$, $n\in\N$, in $S_k(\Zmod{p^m})$ for some weight $k$.

We also denote by $S_{\le k}(\Zmod{p^m})$ the direct sum over all $j\le k$ of the spaces $S_j(\Zmod{p^m})$.
\end{dfn}

We recall below in section \ref{subsection:definition_of_weight_bounds} an example due to Calegari and Emerton showing that there may be an infinite number of weak eigenforms in a space $S_k(\Zmod{p^2})$. Hence the distinction between strong and weak eigenforms modulo $p^m$ is important.
\smallskip

Consider now the following finiteness statements and recall that $p \nmid N$:
\medskip

\noindent $\mathbf{Strong}_m$: The set $\fS_m(N)$ is finite, i.e., the set of strong eigenforms modulo $p^m$ on $\Gamma_1(N)$ is finite.
\medskip

\noindent $\mathbf{Rep}_m$: The set $\fR_m(N)$ is finite, i.e., the set of reductions modulo $p^m$ of the characters of the Galois representations of Hecke eigenforms in level $N$ is finite.
\medskip

\noindent $\mathbf{Weak}_m$: There exists a constant $k = k(N,p,m)$ depending only on $N$, $p$, $m$, such that $\fS_m(N) \subseteq S_{\le k}(\Zmod{p^m})$, i.e., any strong eigenform modulo $p^m$ occurs in the weak sense at weight at most $k$.
\medskip

\noindent $\mathbf{Fin}_m$: For any $k$ the set $\fS_m(N) \cap S_{\le k}(\Zmod{p^m})$ is finite, i.e., for any $k$ there is only a finite number of strong eigenforms occurring in the weak sense at weight $k$.
\medskip

We also consider the following finiteness condition that was raised as a question by K.\ Buzzard at least for $\Gamma_0(N)$ (\cite{buzzard}, Question~4.4). Recall again $p \nmid N$.
If $f\in \fS(N)$, denote by $K_{f,p}$ the $p$-adic coefficient field of $f$:
$$
K_{f,p} = \Q_p ( a_{\ell}(f) \mid ~ \mbox{$\ell$ prime with $\ell\nmid Np$}).
$$

The finiteness condition then concerns the degree $[K_{f,p} : \Q_p]$:
\medskip

\noindent $\mathbf{(B)}$: There is a constant $c(N,p)$ depending only on $N$, $p$, such that $[K_{f,p} : \Q_p] \le c(N,p)$ for all $f\in \fS(N)$.
\medskip

Additionally, we will be considering a finiteness condition $\mathbf{I}_m$ that asserts a uniform bound for all $f\in \fS(N)$ for the index of the projections to $\Zmod{p^m}$ of $\Z_p [ a_{\ell}(f) \mid ~ \mbox{$\ell$ prime with $\ell\nmid Np$}]$ inside its normalization. See \ref{subsection:I_m} below for a precise statement of the condition.
\medskip

Our most basic finiteness conjecture concerns strong eigenforms modulo $p^m$. We will see further below that the following two conjectures are in fact equivalent.

\begin{conj}\label{conj:strong_weight_bounds}
$\mathbf{Strong}_m$ holds for any $m$.
\end{conj}

\begin{conj}\label{conj:finitely_many_strong_at_fixed_weight}
$\mathbf{Fin}_m$ holds for any $m$.
\end{conj}

\subsection{Relation between the conjectures and results.} It is immediately clear that statement $\mathbf{Strong}_m$ implies each of the statements $\mathbf{Rep}_m$, $\mathbf{Weak}_m$, and $\mathbf{Fin}_m$.

By work of Jochnowitz~\cite{jochnowitz_finiteness}, statement $\mathbf{Strong}_1$ is true, and hence so are $\mathbf{Rep}_1$, $\mathbf{Weak}_1$, and $\mathbf{Fin}_1$.

But already the question of whether $\mathbf{Strong}_2$ holds seems to be totally open; in fact, a result in this paper (Theorem~\ref{thm:finiteness_conditions}) suggests that $m=2$ is already the decisive case.
\smallskip

In section~\ref{section:buzzard} we prove the following relationship between some of the above finiteness conditions. In our opinion, this sheds light on some aspects of Buzzard's question, i.e., whether condition $\mathbf{(B)}$ holds, and it places our finiteness conjectures into a framework that has already attracted some attention (see for instance Conjecture~1.1 of~\cite{calegari_emerton_large_index} that implies (and is conjecturally equivalent to) statement $\mathbf{(B)}$.)

\begin{thm}\label{thm:finiteness_conditions} The following are equivalent:
\begin{enumerate} \item Statement $\mathbf{(B)}$ holds, and so in particular Buzzard's question has an affirmative answer.
\item For all $m\in\N$, statements $\mathbf{Strong}_m$ (or $\mathbf{Rep}_m$) and $\mathbf{I}_m$ hold.
\item Statements $\mathbf{Strong}_2$ (or $\mathbf{Rep}_2$) and $\mathbf{I}_2$ hold.
\end{enumerate}
\end{thm}

In a previous version of this article, the following theorem was stated as a conjecture. We are indebted to Frank Calegari for explaining the proof to us, cf.\ the additional remarks below in subsection \ref{subsection:summary,N=1,p=2}.

\begin{thm}\label{thm:weak_weight_bounds} Assume $p \ge 5$. Then the statement $\mathbf{Weak}_m$ holds for any $m$, that is, there exists a constant $k = k(N,p,m)$ depending only on $N$, $p$, $m$, such that any strong eigenform modulo $p^m$ occurs in the weak sense at weight at most $k$ in the same level~$N$.
\end{thm}

\begin{rem}\label{rem:weak_weight_bounds}
Without the assumption $p\ge 5$, the following slightly weaker statement is true: there exists a constant $k' = k'(N,p,m)$ depending only on $N$, $p$, $m$, such that for any strong eigenform~$f$ modulo $p^m$ there is a weak eigenform~$g$ mod~$p^m$ at weight at most $k$ in the same level~$N$ such that $f$ and $g$ agree at all coefficients of index prime to~$p$.

The only reason we need the assumption $p\ge 5$ occurs in the very last step of the proof where we need a ``weight bound'' in connection with the application of the $\theta$ operator modulo $p^m$. Such a bound is provided by \cite{ck} that avoids a discussion of the cases $p\in\{ 2,3\}$. Without having worked out the details it is however fairly clear to us that such bounds should exist for all $p$ and that Theorem \ref{thm:weak_weight_bounds} holds without any restriction on $p$.
\end{rem}

As it is immediately clear from definitions that the two conditions $\mathbf{Weak}_m$ and $\mathbf{Fin}_m$ together imply $\mathbf{Strong}_m$, we thus have the equivalence of Conjectures \ref{conj:strong_weight_bounds} and \ref{conj:finitely_many_strong_at_fixed_weight} as a consequence of Theorem \ref{thm:weak_weight_bounds}.

\begin{cor} Conjectures \ref{conj:strong_weight_bounds} and \ref{conj:finitely_many_strong_at_fixed_weight} are equivalent.
\end{cor}

\subsection{Strong eigenforms and characters.}
We stress that by Definition~\ref{definition:basic_definitions} strong eigenforms modulo~$p^m$ are equal if their $q$-expansions are equal. In particular, the Fourier coefficients at the ``bad primes'' (i.e., those dividing $Np$) of equal strong eigenforms modulo~$p^m$ are equal. We wish to formulate the following conjecture:

\begin{conj}\label{conj:repr_m} Statement $\mathbf{Rep}_m$ holds for all $m$.
\end{conj}

It is clear that this conjecture is weaker than Conjecture \ref{conj:strong_weight_bounds}.
We think that it is an interesting question whether they are in fact equivalent. We have not attempted to answer this question. In fact, the (necessarily continuous) character of a Galois representation modulo~$p^m$ is uniquely determined by the images of $\Frob_\ell$ for primes~$\ell$ in any density-one set of primes, hence it does not give any information (a priori) on the coefficients at ``bad primes'' of the strong eigenforms modulo~$p^m$ admitting this representation.

We would like to add a word of explanation why we chose the set $\fR_m(N)$ the way we did: ultimately we would like to understand the set of isomorphism classes of strongly modular Galois representations modulo~$p^m$; if its residual (i.e., the mod~$p$) representation is absolutely irreducible, the isomorphism class of the representation modulo~$p^m$ is uniquely determined by its character, cf.\ \cite[Th\'{e}or\`{e}me 1]{carayol}; however, this does not necessarily hold if the residual representation is not absolutely irreducible. Our choice of the set $\fR_m(N)$ is thus explained by our desire not to enter into details about the possible reductions modulo~$p^m$ of Galois representations which are not residually absolutely irreducible.

\subsection{Reformulations in terms of weight bounds.}\label{subsection:weight_bounds}
Let $f$ be an eigenform modulo $p^m$, either strong or weak. We say that $f$ occurs strongly resp.\ weakly at a specific weight $k_0$ if there is a strong resp.\ weak eigenform in $S_{k_0}(\Gamma_1(N),\Zmod{p^m})$ equal to~$f$.

Thus, Conjecture~\ref{conj:strong_weight_bounds} can now equivalently be formulated as the statement that there be a constant $\kappa(N,p,m)$ depending only on $N$, $p$, $m$ such that any strong eigenform modulo $p^m$ on $\Gamma_1(N)$ occurs strongly at weight $\le \kappa(N,p,m)$. We call such a constant -- if it exists -- a {\it strong weight bound} for strong eigenforms modulo $p^m$ on $\Gamma_1(N)$. Similarly, the constant from Theorem~\ref{thm:weak_weight_bounds} is called a {\it weak weight bound} for strong eigenforms modulo $p^m$ on $\Gamma_1(N)$. Hence, Conjecture~\ref{conj:strong_weight_bounds} and Theorem~\ref{thm:weak_weight_bounds} state the existence for any $m$ of a strong resp.\ weak weight bound for strong eigenforms modulo $p^m$ on $\Gamma_1(N)$.

One can ask whether a stronger form of Theorem~\ref{thm:weak_weight_bounds} holds: does any weak eigenform modulo $p^m$ occur weakly at some weight bounded by a function of $N$, $p$, $m$? However, in this paper we chose not to consider this strengthening.

\subsection{Relation between strong and weak eigenforms mod \texorpdfstring{$p^m$}{pm} and between Conjecture \ref{conj:strong_weight_bounds} and Theorem \ref{thm:weak_weight_bounds}.}
In \cite{ckw} two of us with Imin Chen introduced the notions of strong and weak eigenforms modulo $p^m$ with a slightly different definition. We include in section~\ref{section:weak_not_strong} an explicit example showing that weak eigenforms are not necessarily strong (at any weight). This is not the only indication that Conjecture \ref{conj:strong_weight_bounds} probably does not follow from Theorem \ref{thm:weak_weight_bounds} in any immediately obvious way. Even more serious is the fact that there may exist infinitely many weak eigenforms modulo $p^m$ at some fixed weight, as discovered by Calegari and Emerton, see section~\ref{section:abundance_weak}.

\subsection{The special case \texorpdfstring{$N=1$}{N=1} and \texorpdfstring{$p=2$}{p=2}.}\label{subsection:summary,N=1,p=2} Given Theorem \ref{thm:weak_weight_bounds} above, it is natural to ask about the form of the constant $k(N,p,m)$. In the very special case where $N=1$, $p=2$, and with the further restriction that there be no ramification in the coefficient fields, we shall show that one can use the theory of Serre and Nicolas to obtain explicit bounds. The relevant theorem is Theorem \ref{thm:weightbound}, see also Remark \ref{remark:thm_N=1,p=2} below. We do not obtain ``formulas'' for the weight bounds in question, merely a method for computing such. We illustrate the method with the computation of explicit bounds for the cases $m\le 4$ at the end of subsection \ref{subsection:weak weight bounds mod 2^m}.

As Frank Calegari explained to us, the structure of the core inductive step in the proof of Theorem \ref{thm:weightbound} can be reinterpreted in such a way that it is susceptible to generalization by using a decisive input from \cite{calegari_emerton_large_index}, specifically \cite[Theorem 2.2]{calegari_emerton_large_index}. He then showed us a sketch of Theorem \ref{thm:weak_weight_bounds} above. The proof of Theorem \ref{thm:weak_weight_bounds} that we give is a slight variation of the sketch that we owe to Calegari, in particular we shall work with group cohomology rather than the cohomology of modular curves as in \cite{calegari_emerton_large_index}.

As Theorem \ref{thm:weightbound} is by now superseded by Theorem \ref{thm:weak_weight_bounds}, we have chosen to skip some of the details of its proof. For full details, the interested reader is referred to the preprint version of this paper, see \cite{krw_preprint}.

\subsection{Numerical data.} In section~\ref{subsection:numerical_data} we provide a bit of numerical data pertaining to this question of how the constants $k(N,p,m)$ of Theorem \ref{thm:weak_weight_bounds} depend on $N$, $p$, $m$. Our data set is not sufficiently large to warrant any conjectures about the optimal shape of the constants $k(N,p,m)$. However, the data do raise the interesting question of whether these constants can be chosen so as to be independent of $N$.

In the spirit of Serre's Modularity Conjecture, it seems reasonable to ask whether a weight bound can be derived purely locally at~$p$, in which case the independence of~$N$ would be clear. One could for instance try to classify the reductions modulo~$p^m$ of all $2$-dimensional crystalline representations of the absolute Galois group of~$\Q_p$, and to prove that all of them can be obtained as reductions coming from bounded weight. For $m=1$, this is done in \cite{BhGh} for slopes between $1$ and~$2$, in \cite{BhGhRo} for slope~$1$ and in \cite{BuGe} for slopes less than~$1$. Moreover, one might then hope to have some local-global compatibility modulo~$p^m$ in order to obtain a weak eigenform (a global object) in the minimal weight predicted locally. For the authors, this appears only as a vague speculation at the moment. If one had classification results of reductions modulo~$p^m$ of crystalline representations for given slope (or range of slopes) like those cited above for $m=1$, one might be able to check whether such a local-global compatibility could hold, and possibly make precise predictions for weight bounds. For this, also see the next paragraph.

\subsection{The finiteness conjectures for finite slope.} One test of Conjecture \ref{conj:strong_weight_bounds} is to ask whether it holds if we restrict the eigenforms in some way. If one restricts to eigenforms with a fixed, finite $p$-slope, the finiteness statement of Conjecture~\ref{conj:strong_weight_bounds} becomes true, but then additional questions arise in the direction of a more precise, ``quantitative'' version of Conjecture~\ref{conj:strong_weight_bounds}. We will discuss this briefly in section~\ref{subsection:weight_bounds_fixed_slope}.

\subsection{Motivation}\label{subsection:motivation} We now explain the motivation underlying this work by first considering the ``level'' aspect before treating the ``weight'' one, which is the focus of this paper. Let $q \mid N$ be a prime different from~$p$. Ribet's famous theory of ``level lowering'', which is a fundamental input in the proof of Fermat's Last Theorem and of Serre's Modularity Conjecture, translates statements on the structure of $\rhobar_{f,p}|G_q$, with $G_q = G_{\Q_q} = \Gal(\Qbar_q/\Q_q)$, into a congruence modulo~$p$ of $f$ and some other Hecke eigenform at level $N/q$; for instance, if $q \mid\mid N$, then $\rhobar_{f,p}|G_q$ is unramified at~$q$ if and only if such a congruence exists. In the quest of determining (computationally or theoretically) the structure of $\rho_{f,p}|G_q$ it is thus most natural to relate statements on $\rho_{f,p,m}|G_q$ to statements on congruences modulo~$p^m$ of Hecke eigenforms. The underlying theory of ``level lowering modulo~$p^m$'' has to some extent been developed, especially by Dummigan~\cite{dummigan} and also by Tsaknias~\cite{tsaknias}, but there are still many open cases. For an application of level lowering modulo higher powers of $p$ to Diophantine problems, see \cite{dahmen_yazdani}.

We now turn our attention to weights. By the weight aspect of Serre's Modularity Conjecture, i.e., the theorem of Khare and Wintenberger, there is a {\em minimal weight} determined by the restriction $\rhobar_{f,p}|G_p$ (even by the restriction to the inertia group at~$p$) such that in that weight there is a Hecke eigenform~$g$ of the same level as~$f$ such that $f$ and $g$ are congruent modulo~$p$; conversely, such a congruence determines the shape of $\rhobar_{f,p}|G_p$. It is thus natural to approach the study of $\rho_{f,p}|G_p$ through approximations modulo~$p^m$ on the modular side, i.e., through congruences modulo~$p^m$ with forms in ``low'' weights.

Finally it is worthwhile to mention the question of the existence and number theoretic meaning of companion forms modulo~$p^m$ because it is also situated in the spirit of weights modulo~$p^m$ (see \cite{AdMa}, that, however, is restricted to ordinary forms and coefficients unramified at~$p$).

\section{Proofs of the theorems}\label{section:buzzard} Before beginning the proof of Theorem \ref{thm:finiteness_conditions} we first make some initial observations concerning coefficient fields and modular Galois representations modulo $p^m$. We also introduce the finiteness statement $\mathbf{I}_m$ in detail.

\subsection{Coefficient fields} For $f\in\fS(N)$ we have the coefficient field
$$
K_f := \Q(a_n(f) \mid ~ n\in\N ),
$$
which is a finite extension of $\Q$. For $D\in\N$ let
$$
K_f^{(D)} := \Q(a_{\ell}(f) \mid ~\mbox{$\ell$ prime with $\ell\nmid ND$} ).
$$

The statement of the following lemma is well-known, but we include the short proof for lack of a precise reference.

\begin{lem}\label{lem:coeffs_ND}
In the above setting we have $K_f^{(D)} = K_f$ for any $D\in\N$.

Consequently, for the $p$-adic coefficient field
$$
K_{f,p} = \Q_p ( a_{\ell}(f) \mid ~ \mbox{$\ell$ prime with $\ell\nmid Np$}),
$$
we have that $K_{f,p} = \Q_p ( a_n(f) \mid ~ n\in\N )$ for $f\in\fS(N)$.
\end{lem}

\begin{proof} The second statement clearly follows from the first. To prove the first statement, since $f$ is an eigenform, it suffices to prove that $a_q\in K_f^{(D)}$ for any prime $q$ dividing~$ND$. Let $q$ be such a prime and let $\sigma \in \Gal(\Qbar/K_f^{(D)})$. Now,
$$
\sigma f := \sum_n \sigma(a_n) q^n
$$
is again an newform on $\Gamma_1(N)$ with nebentypus $\sigma \epsilon$ if $\epsilon$ is the nebentypus of $f$, cf.\ \cite[Proposition 2.7]{deligne-serre}.

If $\ell$ is any prime not dividing $ND$ we have $\sigma a_{\ell} = a_{\ell}$. As $f$ and $\sigma f$ are eigenforms for all Hecke operators, this implies that $\sigma a_t = a_t$ for all $t\in\N$ with $(t,ND) = 1$. By Multiplicity One, see specifically \cite[Theorem 4.6.19]{miyake}, we can conclude that $\sigma f = f$ whence in  particular that $\sigma a_q = a_q$.

Since this holds for any $\sigma \in \Gal(\Qbar/K_f^{(D)})$, we must have $a_q\in K_f^{(D)}$.
\end{proof}

Another consequence of the Lemma is that for $f\in\fS(N)$ all character values $\epsilon(d)$, where $\epsilon$ is the Dirichlet character/nebentype of~$f$, lie in the field $K_f^{(D)}$: by \cite[Corollary 3.1]{ribet_nebentypus}, one knows that they are all in $K_f$, and by Lemma \ref{lem:coeffs_ND} we have $K_f = K_f^{(D)}$.

\begin{rem} Lemma \ref{lem:coeffs_ND} is false in general if we just assume that $f$ is an eigenform. A concrete counterexample is as follows. Consider $f:=\Delta$, the unique form of weight $12$ and level $1$. At level $3$, $f$ gives rise to two oldforms, $f$ and $g:=f(q^3)$. If one computes the action of $U_3$ (the Hecke operator corresponding to the prime $3$ at level $3$) on the basis $f,g$ of the space of oldforms, one finds that it is given by the matrix $\abcd{a_3}{-3^{11}}{1}{0}$ where $a_3 = a_3(f) = 252$. The characteristic polynomial of this is $x^2 - 252x + 3^{11}$ that is irreducible over $\Q$. Let $\gamma$,$\gamma'$ be the roots. Then $f - \gamma' g$ is a normalized eigenform with the property that the $T_{\ell}$-eigenvalues are in $\Q$ for all primes $\ell\neq 3$, whereas the $U_3$-eigenvalue is $\gamma$ that is not in $\Q$, but rather in a quadratic extension.

This examples reflects a general phenomenon: suppose that our level $N$ has form $N=M\ell^r$ where $\ell$ is prime, and suppose that $f$ is a newform at level $M$. Then $f$ gives rise to oldforms $f(q),f(q^{\ell}),\ldots,f(q^{\ell^r})$, and one can easily and explicitly compute the action of the level $N$ Hecke operator $U_{\ell}$ on the span of these oldforms, see e.g.\ \cite[Proposition 4]{wiese_dihedral}. One finds that its characteristic polynomial equals $(x^2 - a_\ell(f) x + \delta\epsilon(\ell) \ell^{k-1})\cdot x^{r-1}$ where $\epsilon$ is the nebentypus, $k$ is the weight, $\delta = 1$ if $\ell \nmid M$, and $\delta = 0$ otherwise. This proves that the $\ell$-th coefficient of any eigenform in this span lies in an at most quadratic extension of $\Q(a_\ell(f))$.

Using the commutativity of the Hecke operators, we obtain that the field of coefficients of any eigenform is the composite of the coefficient field of the underlying newform with an at most quadratic extension for each prime dividing the level.

An implication of this is the following. If we had defined the set $\fS(N)$ using eigenforms rather than newforms, then the resulting condition $\mathbf{(B)}$ would be equivalent to the condition $\mathbf{(B)}$ that we formulated in \ref{subsection:finiteness_conditions} above: this is because any $p$-adic field admits only finitely many quadratic extensions.

Because of some technicalities in the proof of Theorem \ref{thm:finiteness_conditions}, specifically relating to the finiteness condition $\mathbf{I}_m$, we are confining ourselves to newforms at least as far as that theorem is concerned. In later parts of the paper, such as in Theorem \ref{thm:weak_weight_bounds} or section \ref{subsection:weight_bounds_fixed_slope}, it does not make any difference whether we are working with newforms or eigenforms.
\end{rem}

\subsection{The Galois representation.}\label{subsection:galois_representation} Let $O_{f,p}$ be the valuation ring of $K_{f,p}$, let $\p_{f,p}$ be the maximal ideal of $O_{f,p}$, denote by $e_{f,p}$ the ramification index of $K_{f,p}/\Q_p$, and by $\textrm{res}_{f,p}$ the residue degree.

It follows from the construction of the Galois representation $\rho_{f,p}$ on \'etale cohomology and the Eichler-Shimura theorem that $\rho_{f,p}$ can be defined to take its image in $\GL_2(O_{f,p})$; this involves the choice of a Galois-stable lattice. If the residual representation $\rhobar_{f,p}$ is absolutely irreducible, then by a theorem of Carayol~\cite[Th\'eor\`eme~3]{carayol}, the image can even be taken in $\Z_p[a_\ell(f) \mid ~ \mbox{$\ell$ prime with $\ell\nmid Np$}]$. In that case, the representation is also independent (up to isomorphism) of the chosen lattice; the character is independent in all cases.

Reducing $\rho_{f,p}$ ``mod $p^m$'', as defined in the introduction, means to compose this representation with the reduction of elements in $O_{f,p}$ modulo $\p_{f,p}^{(m-1)e_{f,p}+1}$. The mod~$p^m$ Galois representation $\rho_{f,p,m}$ attached to $f$ has the usual properties such as being unramified outside $Np$ and with $(a_{\ell}(f) \mod{p^m})$ equal to the trace of $\rho_{f,p,m}(\Frob_\ell)$ for any prime $\ell\nmid Np$.

It is important to notice that we have $\mathbf{Rep}_{m+1} \Rightarrow \mathbf{Rep}_m$: any mod $p^m$ representation of the type that we are considering is in fact the reduction mod $p^m$ of some mod $p^{m+1}$ representation $\rho_{f,p,m+1}$.

\subsection{The finiteness statement \texorpdfstring{$\mathbf{I}_m$}{(Im)}.}\label{subsection:I_m} Let again $f\in \fS(N)$. The ring $O_{f,p}$ has the subring $\Z_p[ a_{\ell}(f) \mid ~ \mbox{$\ell$ prime with $\ell\nmid Np$}]$.
In general, the inclusion
$$
\Z_p[a_\ell(f) \mid ~ \mbox{$\ell$ prime with $\ell\nmid Np$}] \subseteq O_{f,p}
$$
is proper, but of finite index. The statement we need is the following ``version modulo~$p^m$'' of it.
Write $\Z/(p^m)[a_\ell(f) \pmod{p^m} \mid \ell\nmid Np\textnormal{ prime}]$ for the subring of $\Zmod{p^m}$ generated by the images of $a_\ell(f)$ for all primes $\ell \nmid Np$; it is naturally a subring of $O_{f,p}/\p^{e_{f,p}(m-1)+1}$.

Now consider for fixed $N$ with $p \nmid N$ the ``index finiteness'' statement:
\smallskip

\noindent $\mathbf{I}_m$: There is a constant $\iota(N,p,m)$ depending only on $N$, $p$, $m$, such that
$$
[ (O_{f,p} / \p_{f,p}^{e_{f,p}(m-1)+1}): ((\Z/(p^m))[a_\ell(f) \pmod{p^m}) \mid \ell\nmid Np \textnormal{ prime}] \le \iota(N,p,m)
$$
for all $f\in\fS(N)$.
\smallskip

It is obvious that $\mathbf{I}_{m+1}$ implies $\mathbf{I}_m$.

In connection with the condition $\mathbf{I}_m$, the reader should be reminded of the following. Let $O_f$ denote the ring of integers of the field $K_f = \Q(a_n(f) \mid ~ n\in\N )$ of coefficients of~$f$. It has been known for a long time that
$$
\sup_{f\in \fS(N)} [ O_f : \Z[a_\ell(f) \mid ~ \mbox{$\ell$ prime with $\ell\nmid Np$}] ] = \infty,
$$
cf.\ \cite[Theorem~1.2]{jochnowitz_index} (that states that this index actually converges to $\infty$ with growing weight). See also \cite[Theorem~2.1]{calegari_emerton_large_index}. Thus, there is certainly no ``global'' reason for why statement $\mathbf{I}_m$ should be true.
\smallskip

We note that since statement $\mathbf{Strong}_1$ holds due to the work of Jochnowitz \cite{jochnowitz_finiteness}, statement $\mathbf{I}_1$ is seen to be equivalent to the residue degrees $\textrm{res}_{f,p}$ being bounded when $f$ runs through $\fS(N)$. It seems to be an open question whether this actually holds. Of course, it is implied by statement $\mathbf{(B)}$.

\subsection{Proof of Theorem \ref{thm:finiteness_conditions}.} We now establish relations between the various finiteness statements and in particular prove Theorem~\ref{thm:finiteness_conditions}.

Let us consider the following inclusions for $f \in \fS(N)$:
\begin{equation}\label{eq:rings}\begin{split}
          & \Z_p/(p^m)[a_{\ell}(f) \pmod{p^m}\;|\; \ell \nmid Np \textnormal{ prime }] =: A_m(f)\\
\subseteq\;\;\; & \Z_p/(p^m)[a_n(f) \pmod{p^m} \;|\; n \in \N] =: B_m(f)\\
\subseteq\;\;\; & O_{f,p}/\p_{f,p}^{e_{f,p}(m-1)+1} =: C_m(f).
\end{split}
\end{equation}

\begin{prop}\label{prop:conj1+2}
Fix $N \in \N$ and let $m \in \N$. We use the notation from equation~\eqref{eq:rings}.
\begin{enumerate}
\item\label{item:1} Statement $\mathbf{Rep}_m$ implies that $\# A_m(f)$ is bounded for $f \in \fS(N)$.
\item\label{item:2} Statement $\mathbf{Strong}_m$ implies that $\# B_m(f)$ is bounded for $f \in \fS(N)$.
\item\label{item:3} Statement $\mathbf{(B)}$ implies statement $\mathbf{I}_m$ and that $\# C_m(f)$ is bounded for $f \in \fS(N)$.
\item\label{item:4} Statement $\mathbf{(B)}$ is equivalent to $\# C_2(f)$ being bounded for $f \in \fS(N)$.
\item\label{item:5} Statement $\mathbf{(B)}$ implies statement $\mathbf{Rep}_m$.
\item\label{item:6} The conjunction of statement $\mathbf{Rep}_m$ and the boundedness for $f \in \fS(N)$ of the index $A_m(f) \subseteq B_m(f)$ is equivalent to statement $\mathbf{Strong}_m$.
\item\label{item:7} The conjunction of statements $\mathbf{Rep}_m$ and $\mathbf{I}_m$ implies the boundedness of $\#C_m(f)$ for $f \in \fS(N)$.
\item\label{item:8} Statement $\mathbf{(B)}$ implies statement $\mathbf{Strong}_m$.
\end{enumerate}
\end{prop}

\begin{proof}
\eqref{item:1} This is clear because there are only finitely many different characters modulo~$p^m$.

\eqref{item:2} This is equally clear because there are only finitely many different strong eigenforms modulo~$p^m$.

\eqref{item:3} $\mathbf{(B)}$ implies that $e_{f,p}$ and $\textrm{res}_{f,p}$ are bounded for $f \in \fS(N)$, hence so is $\#C_m(f) = \#O_{f,p}/\p_{f,p}^{e_{f,p}(m-1)+1}=p^{\textrm{res}_{f,p}(e_{f,p}(m-1)+1)}$. This implies that the fraction $\frac{\#C_m(f)}{\#A_m(f)}$, which equals the index in question, is also bounded, and we thus get statement $\mathbf{I}_m$.

\eqref{item:4} If $\#C_2(f) = p^{\textrm{res}_{f,p}(e_{f,p}+1)}$ is bounded for $f \in \fS(N)$, then
so are $\textrm{res}_{f,p}$ and $e_{f,p}$, implying $\mathbf{(B)}$.

\eqref{item:5} As we are assuming $\mathbf{(B)}$ there is a finite extension $K$ of $\Q_p$ of bounded degree such that any $f\in \fS(N)$ has coefficients in~$O$, the valuation ring of~$K$. If $\p$ is the prime of $O$ above $p$ then this means that any representation $\rho_{f,p,m}$ attached to an $f\in \fS(N)$ as above has image in the finite group $G := \GL_2(O/\p^{\gamma})$, where $\gamma := e(K/\Q_p)(m-1) + 1$. Hence, the degree $[L:\Q]$ is bounded, where $L/\Q$ is the extension cut out by a $\rho_{f,p,m}$. As $L$ ramifies at most at $Np$, by Hermite-Minkowski (and the fact that bounded degree and bounded ramification set imply bounded discriminant) there are only finitely many possibilities for $L/\Q$. For each $L/\Q$ there are only finitely many equivalence classes of representations $\Gal(L/\Q) \hookrightarrow \GL_2(O/\p^{\gamma})$, and thus $\fR_m(N)$ is finite, i.e., we have statement $\mathbf{Rep}_m$.

\eqref{item:6} Statement $\mathbf{Strong}_m$ implies $\mathbf{Rep}_m$ and the boundedness of $\#B_m(f)$ and so in particular the boundedness of the index $A_m(f) \subseteq B_m(f)$. Conversely, by $\mathbf{Rep}_m$ there are only finitely many collections of numbers $(a_\ell(f) \pmod{p^m})$ for $\ell \nmid Np$. Moreover, $\mathbf{Rep}_m$ implies the boundedness of $\#A_m(f)$, and hence together with the boundedness of the index $A_m(f) \subseteq B_m(f)$ implies the boundedness of $\#B_m(f)$. So, for each prime $\ell$ with $\ell \mid Np$ there is then only a finite number of possibilities for $(a_\ell(f) \pmod{p^m})$. We deduce that $\mathbf{Strong}_m$ holds.

\eqref{item:7} This is clear from the inclusions \eqref{item:1}.

\eqref{item:8} As statement $\mathbf{(B)}$ implies both $\mathbf{Rep}_m$ by~\eqref{item:5} and the finiteness
of the index $A_m(f) \subseteq B_m(f)$, the result follows by~\eqref{item:6}.
\end{proof}

\begin{proof}[Proof of Theorem~\ref{thm:finiteness_conditions}.] Statement $\mathbf{(B)}$ implies $\mathbf{Strong}_m$, $\mathbf{Rep}_m$, and $\mathbf{I}_m$ by Proposition~\ref{prop:conj1+2}\eqref{item:8}, \eqref{item:5}, and \eqref{item:3}. Thus $(1) \Rightarrow (2)$ in the Theorem is clear.

We also recall the trivial implications $\mathbf{Strong}_m \Rightarrow \mathbf{Rep}_m$, $\mathbf{Rep}_{m+1} \Rightarrow \mathbf{Rep}_m$, and $\mathbf{Strong}_{m+1} \Rightarrow \mathbf{Strong}_m$.

Thus, if we assume $\mathbf{Rep}_m$ and $\mathbf{I}_m$ for some $m\ge 2$, we have $\mathbf{Rep}_2$ and $\mathbf{I}_2$ and hence the boundedness of $\#C_2(f)$ for $f \in \fS(N)$ by Proposition~\ref{prop:conj1+2}\eqref{item:7}; we then obtain $\mathbf{(B)}$ by Proposition~\ref{prop:conj1+2}\eqref{item:4}. Statement $\mathbf{Strong}_m$ for all $m$ follows.

This shows that we have $(2) \Rightarrow (3) \Rightarrow (1)$ in the Theorem.
\end{proof}

\subsection{Weak eigenforms modulo \texorpdfstring{$p^m$}{pm} are not necessarily strong in any weight if \texorpdfstring{$m>1$}{m>1}.}\label{section:weak_not_strong}
This was raised as a question in~\cite{ckw}. For an explicit example, let $f = E_4^6 \Delta + 2\Delta^3$ in weight~$36$, level~$1$ modulo~$4$. One can check that $f$ is an eigenform by computing the first couple of Hecke operators. Note that $f$ is not congruent to $\Delta$ modulo~$4$, so by the result of Hatada that every strong level~$1$ form modulo~$4$ has to be~$\Delta$ (Theorems 3 and~4 of~\cite{hatada1977b}), we find that $f$ is weak but not strong.

\subsection{Abundance of weak eigenforms and of Galois representations mod \texorpdfstring{$p^m$}{pm}.}\label{section:abundance_weak}
There exist infinitely many weak eigenforms modulo $p^m$ in some fixed weights. This was pointed out by Calegari and Emerton in \cite{calegari_emerton_large_index}, and the reasoning is as follows.
Suppose that we have eigenforms $f$ and $g$ in some $S_k(\Gamma_1(N),\Zbar_p )$ such that
$$
f\equiv g \pmod{p}, \quad \mbox{but} \quad f \not\equiv g \pmod{p^2} .
$$
Suppose for simplicity that their coefficient fields are unramified over $\Q_p$ so that $p$ generates the maximal ideal of their valuation rings. Let us write the forms like this:
$$
f = \varphi + pf_1 \textnormal{ and } g = \varphi + pg_1
$$
with modular forms $\varphi,f_1,g_1$. The eigenvalues of some Hecke operator $T$ on $f$ and $g$ are
$$
\lambda = \alpha + p\lambda_1 \textnormal{ and } \mu = \alpha + p\mu_1.
$$
Let $a,b$ be in the maximal unramified extension of $\Z_p$ such that $a+b$ is invertible. Put
$$
h_{a,b} = \frac{af+bg}{a+b}.
$$
Note $ph_{a,b} \equiv p\varphi \pmod{p^2}$, independent of $a$ and $b$. This yields
$$
Th_{a,b} \equiv (\alpha + p \frac{a\lambda_1 + b\mu_1}{a+b}) h_{a,b} \pmod{p^2}.
$$
Since $f \not\equiv g \pmod{p^2}$ there exists some Hecke operator $T$ for which the $\lambda_1$ and $\mu_1$ are not both $0$. It is then obvious that the eigenvalue of that $T$ on $h_{a,b}$ can take infinitely many different values by varying $a$ and $b$.

We point out that this implies also that there can be infinitely many non-isomorphic modular Galois representations modulo~$p^2$ in the same level (under the assumption that the mod~$p$ Galois representation attached to~$f$ is absolutely irreducible). Hence, the analogue of Conjecture \ref{conj:strong_weight_bounds} with weak eigenforms mod~$p^m$ instead of strong ones is false. Moreover, by choosing $a,b$ appropriately, the ring of traces can be of arbitrarily large degree.

Of course, the same argument works mutatis mutandis with more general coefficients and also generally in a situation where $f\equiv g \pmod{p^m}$, but $f \not\equiv g \pmod{p^{m+1}}$.

\subsection{Proof of Theorem \ref{thm:weak_weight_bounds}.}\label{subsection:definition_of_weight_bounds}
We now begin to prepare for the proof of Theorem \ref{thm:weak_weight_bounds}. As already mentioned above in \ref{subsection:summary,N=1,p=2}, the idea of the proof was explained to us by Frank Calegari. We shall give a slight variant of his sketch, working in particular with group cohomology rather than cohomology of modular curves.
We allow $p$ to be any prime; the specialisation $p \ge 5$ will only appear in the end, so that we will prove Remark~\ref{rem:weak_weight_bounds} along the way.

Let us fix $k_0 \in \N$ and $m\in\N$. Since the results are known for $m=1$, we henceforth assume $m \ge 2$.

We modify Euler's $\varphi$-function in the following way, motivated by the
weights of the Eisenstein series to be used at the end of the proof. We let $\tilde\varphi(p^m) = \varphi(p^m) = (p-1)p^{m-1}$
(i.e., the usual Euler totient function) if $p > 2$ and $\tilde\varphi(2^m) = 2^{m-2}$.

Define
$$
M_{\le n} := \bigoplus_{2 \le k \le n, k \equiv k_0 \mod \tilde\varphi(p^m)} M_k(\Gamma_1(N), \Qbar_p),
$$
and define $\TT_{\le n}$ as the $\Z_p$-algebra of all Hecke operators, acting diagonally on $M_{\le n}$.
We then define
$$
M := \bigoplus_{k \ge 2, k \equiv k_0 \mod \tilde\varphi(p^m)} M_k(\Gamma_1(N), \Qbar_p),
$$
i.e., $M$ is the direct limit of the $M_{\le n}$. We define $\TT$ as the projective limit of the $\Z_p$-algebras $\TT_{\le n}$. Then $\TT$ acts in a natural way on $M$.

We consider the natural map
$$
M \longrightarrow M_1 := \bigoplus_{k \ge 2, k \equiv k_0 \mod \tilde\varphi(p^m)} M_k(\Gamma_1(N) \cap \Gamma(p^m), \Qbar_p) ,
$$
which is of course an injection. This injection is not equivariant with respect to the Hecke operator~$T_p$ and this is the main reason why we are working with the $p$-deprived Hecke algebras below.

For the sake of completeness, we include a formal argument for the equivariance with respect to the action of the Hecke operators $T_n$ with $n$ prime to~$p$.
It suffices to consider the action of $T_{\ell}$ where $\ell$ is a prime, $\ell\neq p$. Put:
$$
\delta := \abcd{1}{0}{0}{\ell} .
$$
Then, regardless of whether $\Gamma$ is $\Gamma_1(N)$ or $\Gamma_1(N) \cap \Gamma(p^m)$, we find the following representatives $\{ \gamma_i\}$ of left cosets of $\delta^{-1} \Gamma \delta \cap \Gamma$ in $\Gamma$ (cf.\ for instance the reasoning in \cite[5.2]{diamond-shurman}). Put:
$$
\gamma_i := \abcd{1}{p^m i}{0}{1}
$$
for $i=0,\ldots, \ell-1$, and, in case $\ell\nmid N$,
$$
\gamma_{\infty} := \abcd{a\ell}{p^m b}{Np^m}{1}
$$
where $a,b\in\Z$ are chosen so that $a\ell - p^{2m}Nb = 1$. Then with $\delta_i := \delta \gamma_i$ so that
$$
\delta_i = \abcd{1}{p^m i}{0}{\ell}
$$
for $j=0,\ldots, \ell-1$, and
$$
\delta_{\infty} = \abcd{a\ell}{p^m b}{Np^m\ell}{\ell}
$$
the action of $T_{\ell}$ on the weight $k$ component both of $M$ and $M_1$ is given by
$$
f \mapsto \sum_i f\mid_k \delta_i
$$
where the summation is over $i\in \{0,\ldots, \ell-1\}$ if $\ell\mid N$, and over $i\in \{0,\ldots, \ell-1\} \cup \{ \infty\}$ if $\ell\nmid N$.

Next we utilize the Eichler-Shimura embedding of $M_1$ into $H \otimes \Qbar_p$ where
$$
H := \bigoplus_{k \ge 2, k \equiv k_0 \mod \tilde\varphi(p^m)} H^1(\Gamma_1(N) \cap \Gamma(p^m),V_{k-2})
$$
with
$$
V_w := \Sym^w (\Z_p^2)
$$
for $w\in\Z_{\ge 0}$. This embedding is equivariant with respect to the action of the Hecke operators $T_n$ with $n$ prime to $p$. We recall, cf.\ \cite[8.3]{shimura} or \cite{diamond-im}, p.\ 116, the action of a Hecke operator $T_{\ell}$ for $\ell$ prime, $\ell\neq p$, on $H^1(\Gamma_1(N) \cap \Gamma(p^m),V_w)$: if $c\in C^1(\Gamma_1(N) \cap \Gamma(p^m), V_w)$ then
$$
(T_{\ell} c)(\gamma) = \sum_i \delta_i^{\iota} . c(\delta_i \gamma \delta_{j(i)}^{-1})
$$
where $j(i)$ is such that $\delta_i \gamma \delta_{j(i)}^{-1} \in \Gamma$, and $\iota$ is the involution
$$
\abcd{a}{b}{c}{d} ^{\iota} := \abcd{d}{-b}{-c}{a},
$$
and where, as above, the summation is over $i\in \{0,\ldots, \ell-1\}$ if $\ell\mid N$, and over $i\in \{0,\ldots, \ell-1\} \cup \{ \infty\}$ if $\ell\nmid N$.

Let us now write $\TT'$ and $\TT'_{coh}$ for the `$p$-deprived' Hecke algebras over $\Z_p$, i.e., those generated by the operators whose index is prime to~$p$, acting on $M$ and $H$, respectively. As summary of the above discussion we now have the following.

\begin{lem}\label{lem:ce-disc-1} There is a Hecke equivariant injection $M \hookrightarrow H \otimes \Qbar_p$ giving rise to a surjection $\TT'_{coh} \twoheadrightarrow \TT'$ of $\Z_p$-algebras.
\end{lem}

The decisive observation is now the following proposition.

\begin{prop}\label{prop:ce-disc-2} Let $I_{coh}$ be the annihilator of $\TT_{coh}'$ acting on $H \otimes \Z/p^m\Z$.

Then $I_{coh}$ is an open ideal and the quotient $\TT_{coh}'/I_{coh}$ is finite. Hence, if we denote by $I$ the image of the ideal $I_{coh}$ under the surjection $\TT_{coh}' \twoheadrightarrow \TT'$ from Lemma~\ref{lem:ce-disc-1}, then $I$ is an open ideal of~$\TT'$ and the quotient $\TT'/I$ is also finite.
\end{prop}

\begin{proof} Let us start with the remark that annihilators of continuous actions on Hausdorff spaces are always closed ideals. Moreover, in profinite rings, closed ideals of finite index are open.

Retain notation as above, before Lemma \ref{lem:ce-disc-1}, and put $\Gamma := \Gamma_1(N) \cap \Gamma(p^m)$. Notice first that the short exact sequence
$$
0 \longrightarrow \Sym^w (\Z_p^2) =: V_w \stackrel{\cdot p^m}{\longrightarrow} V_w \longrightarrow V_w \otimes \Z/p^m \Z \longrightarrow 0
$$
gives rise to the injection
$$
0 \longrightarrow H^1(\Gamma,V_w) \otimes \Z/p^m \Z \longrightarrow H^1(\Gamma, V_w \otimes \Z/p^m \Z) .
$$
(In fact, this is most often an isomorphism, notably when $N$ is sufficiently large so that $\Gamma$ is torsion free, or when $p\ge 5$; this follows from the proof of Proposition 2.6(a) of \cite{wiese_parabolic}; however, we do not need to know this for our argument.)

We have a natural definition of Hecke operators acting on $H^1(\Gamma, V_w \otimes \Z/p^m \Z)$, namely by the ``same'' formulas that define the action on $H^1(\Gamma,V_w)$. Thus we have a $p$-deprived Hecke algebra $\widetilde{\TT}_{coh}'$ acting on
$$
\widetilde{H} := \bigoplus_{k \ge 2, k \equiv k_0 \mod \tilde\varphi(p^m)} H^1(\Gamma,V_{k-2} \otimes \Z/p^m \Z) ,
$$
and an injection $H \otimes \Z/p^m \Z \hookrightarrow \widetilde{H}$ that is $p$-deprived Hecke equivariant. Thus, we have a surjection $\widetilde{\TT}_{coh}'/\widetilde{I}_{coh} \twoheadrightarrow \TT_{coh}'/I_{coh}$, where $\widetilde{I}_{coh}$ denotes the annihilator of $\widetilde{\TT}_{coh}'$ acting on $\widetilde{H}$. We see that it suffices to show that $\widetilde{\TT}_{coh}'/\widetilde{I}_{coh}$ is finite.

As abelian group, $V_w \otimes \Z/p^m \Z \cong (\Z/p^m \Z)^{w+1}$ is isomorphic to the abelian group of homogeneous polynomials, say in $x$ and $y$, of total degree $w$ and coefficients in $\Z/p^m \Z$. Consider a prime $\ell$ different from $p$. Now the main point is that each of the matrices $\delta_j^{\iota}$ above is congruent modulo $p^m$ to a diagonal matrix with entries $\ell$ and $1$. Hence the action of $\delta_j^{\iota}$ on the basis $x^u y^v$ where $u+v=w$ is diagonal:
$$
\delta_j^{\iota} . x^u y^v = \ell^u x^u y^v
$$
for every $j=0,\ldots,\ell-1, \infty$. Thus, if we view a cocycle $c\in C^1(\Gamma , V_w \otimes \Z/p^m \Z)$ as having values in $(\Z/p^m \Z)^{w+1}$ with coordinate functions $c_i$, $i=0,\ldots w$, then the action of $T_{\ell}$ on $c$ is given in concrete terms as
\begin{eqnarray*}
(T_{\ell} c)(\gamma) &=& \sum_j \delta_j^{\iota} . (c_0(\delta_j \gamma \delta_{i(j)}^{-1}),c_1(\delta_j \gamma \delta_{i(j)}^{-1}), \ldots, c_w(\delta_j \gamma \delta_{i(j)}^{-1}))\\
&=& \sum_j (\ell^w c_0(\delta_j \gamma \delta_{i(j)}^{-1}), \ell^{w-1} c_1(\delta_j \gamma \delta_{i(j)}^{-1}), \ldots , c_w(\delta_j \gamma \delta_{i(j)}^{-1}))\\
&=& (\ell^w \sum_j c_0(\delta_j \gamma \delta_{i(j)}^{-1}) , \ell^{w-1} \sum_j c_1(\delta_j \gamma \delta_{i(j)}^{-1}) , \ldots, \sum_j c_w(\delta_j \gamma \delta_{i(j)}^{-1}))
\end{eqnarray*}
where the sum is over $j=0,\ldots,\ell-1$ when $\ell\mid N$, and over $j=0,\ldots,\ell-1, \infty$ when $\ell\nmid N$.

It follows that, as a module over the $p$-deprived Hecke algebra, we have
$$
H^1 (\Gamma , V_w \otimes \Z/p^m \Z) \cong \bigoplus_{i=0}^w H^1 (\Gamma ,\Z/p^m \Z)(i)
$$
where $H^1 (\Gamma ,\Z/p^m \Z)(i)$ is $H^1 (\Gamma ,\Z/p^m \Z)$, but with twisted Hecke action so that $T_{\ell}$ acts as the old $T_{\ell}$, but then multiplied by $\ell^i$.

Since the largest order of an element in $(\Z/p^m \Z)^{\times}$ is precisely $\tilde\varphi(p^m)$, we conclude that the action of $\widetilde{\TT}_{coh}'/\widetilde{I}_{coh}$ on $\widetilde{H}$ factors through the finite Hecke module
$$
\bigoplus_{i=0}^{\tilde\varphi(p^m)-1} H^1 (\Gamma ,\Z/p^m \Z)(i) ,
$$
and we are done.
\end{proof}

\begin{lem}\label{lem:ce-disc-3} Let $f$ be a normalized eigenform on $\Gamma_1(N)$ of weight~$k_0$ with eigenvalues $a_\ell \in \OO$ for all primes $\ell \neq p$, where $\OO$ is the valuation ring of a finite extension $K$ of~$\Q_p$. Write $e$ for the ramification index of $K/\Q_p$. Then there is a ring homomorphism
$$
\psi: \TT_{coh}'/I_{coh} \to \OO/\p_K^{e(m-1)+1}
$$
such that $\psi(T_\ell) \equiv a_\ell \mod \p_K^{e(m-1)+1}$ for all primes $\ell \neq p$. Moreover, $\psi$ factors through $\TT'/I$.
\end{lem}

\begin{proof} The eigenform~$f$ lives in $M$ and under the injection from Lemma~\ref{lem:ce-disc-1} gives rise to an eigenclass $0 \neq c \in H \otimes K$. The eigenvalues can be expressed by the ring homomorphism
$$
\psi : ~ \TT_{coh}' \twoheadrightarrow \TT' \to \OO
$$
such that $\psi(T_\ell) = a_\ell$ and hence $T_\ell c= a_\ell c$ for all primes $\ell \neq p$. By multiplying or dividing by a power of $\pi$, a fixed uniformizer of~$K$,
we may assume that $c \in H \otimes \OO$ and that $\pi \nmid c$. Write $\overline{c}$ for the image of $c$ in
$$
H \otimes \OO/\p_K^{e(m-1)+1} \cong (H \otimes \Z/p^m\Z) \otimes \OO/\p_K^{e(m-1)+1}.
$$
Denote by $\overline{\psi}$ the composition of $\psi$ with the natural projection $\OO \twoheadrightarrow \OO/\p^{e(m-1)+1}$. The eigenvalue of the Hecke operator~$T$ on the eigenclass $\overline{c}$ is given by $\overline{\psi}(T)$. Let now $T \in I_{coh}$. Then
$$
0 = T\overline{c} = \overline{\psi}(T)\overline{c}.
$$
As $\pi \nmid c$, we conclude $\overline{\psi}(T)=0$, and hence that $\psi$ factors through $\TT_{coh}'/I_{coh}$, as claimed. As $\psi$ factors through~$\TT'$, it follows that $\psi$ factors through~$\TT'/I$.
\end{proof}

\begin{proof}[Proof of Theorem~\ref{thm:weak_weight_bounds}.] We must show the existence of a constant $k(N,p,m)$ depending only on $N$, $p$, $m$, such that any normalized eigenform $f$ on $\Gamma_1(N)$ of any weight $k$ is congruent modulo $p^m$ to a weak eigenform modulo $p^m$ on $\Gamma_1(N)$ in weight bounded by $k(N,p,m)$.

Obviously we may, and will, assume that $k\ge 2$ and only consider such $f$ of weight $k$ congruent modulo $\tilde\varphi(p^m)$ to some fixed $k_0$ (that can be taken to satisfy $2\le k_0 \le 1 + \tilde\varphi(p^m)$.)

Let $f = \sum_n a_n q^n$ be the $q$-expansion of $f$ where the coefficients are in a finite extension $K/\Q_p$.
By Lemma~\ref{lem:ce-disc-3}, $f$ gives rise to a ring homomorphism
$$
\psi: \TT'/I \to \OO/\p_K^{e(m-1)+1}
$$
such that $\psi(T_\ell) \equiv a_\ell \mod \p_K^{e(m-1)+1}$ for all primes $\ell \neq p$. Now, by Proposition \ref{prop:ce-disc-2}, $\TT'/I$ is finite.

For positive integers $s \ge r$, consider the natural surjection $\pi_{r,s}: \TT'_{\le s} \twoheadrightarrow \TT'_{\le r}$ given by restricting the Hecke operators. The $p$-deprived Hecke algebra $\TT'$ is the projective limit of the $\TT'_{\le r}$ with transition maps $\pi_{r,s}$; it comes with natural surjections $\pi_r: \TT' \twoheadrightarrow \TT'_{\le r}$. We further have that the open ideal $I$ of~$\TT'$ is the projective limit of the $\pi_r(I)$. By the exactness of projective limits on compact topological groups, we obtain that $\TT'/I$ is the projective limit of $\TT'_{\le r}/\pi_r(I)$. As $\TT'/I$ is finite, it follows that there is some $r \in \N$ bounded by a constant, depending only on $N$, $p$, and $m$, and with the property that $\TT'_{\le r}/\pi_r(I) \cong \TT'/I$. Hence, $\psi$ can be seen as a ring homomorphism $\psi: \TT'_{\le r}/\pi_r(I) \to \OO/\p_K^{e(m-1)+1}$.

We may then consider $\psi$ as a linear combination of modular forms of weights at most~$r$ that are all congruent to $k_0$ modulo $\tilde\varphi(p^m)$. By multiplying by convenient classical Eisenstein series (cf.\ for instance \cite{serre_p_adic}, p.\ 196) we can actually bring all involved forms into weight~$r$. This is the main point where we use the modified Euler function $\tilde\varphi$: the Eisenstein series in question have weight $\tilde\varphi(p^m) := \varphi(p^m) = p^{m-1}(p-1)$ when $p>2$, but weight $\tilde\varphi(2^m) := 2^{m-2}$ when $p=2$.

It follows that there is a modular form $g = \sum_n b_n q^n$ of weight $r$ on $\Gamma_1(N)$ which modulo $p^m$ is a weak eigenform outside $p$ and such that
$$
b_{\ell} \equiv a_{\ell} \pmod{p^m}
$$
for primes $\ell$ different from $p$. The form $g$ is normalized since $\psi$ is a ring homomorphism. At this point we have proved Remark~\ref{rem:weak_weight_bounds}.

By Corollary \ref{cor:strong_weight_bounds_bounded_slope} in subsection \ref{subsection:weight_bounds_fixed_slope} below, we may further assume that the $p$-slope $v_p(a_p)$ of $f$ exceeds $m-1$ so that $a_p \equiv 0 \pmod{p^m}$.

Let us now assume $p \ge 5$. Possibly we have $b_p \not\equiv 0 \pmod{p^m}$. However, let us consider the result $h = \sum_n c_n q^n$ of applying the $\theta$ operator modulo $p^m$ (with effect $\sum_n \alpha_n q^n \mapsto \sum_n n\alpha_n q^n$ on $q$-expansions) $m$ times to $g$. By \cite[Theorem 1]{ck} we know that $h$ is the reduction modulo $p^m$ of a modular form on $\Gamma_1(N)$ of weight $r + m \cdot (2 + 2p^{m-1}(p-1))$. By construction we will have $c_n \equiv b_n \equiv a_n \pmod{p^m}$ for $n$ coprime to $p$, and also $c_p \equiv 0 \equiv a_p \pmod{p^m}$. Also, $h$ is normalized since $g$ is. We conclude that $h$ is a weak eigenform on $\Gamma_1(N)$ of weight $r + m \cdot (2 + 2p^{m-1}(p-1))$, and that in fact $h\equiv f \pmod{p^m}$. Since $r$ was bounded by a function only depending on $N$, $p$, $m$, our claim follows.
\end{proof}

\section{Further results}\label{section:further_results}

\subsection{Strong weight bounds for bounded \texorpdfstring{$p$-slope}{p-slope}}\label{subsection:weight_bounds_fixed_slope} As before we fix $N \in \N$ with $p\nmid N$. Let
$$
f = \sum a_n q^n \in S_k(\Gamma_1(N),\Zbar_p )
$$
be an eigenform.
We can embed $f$ into the space $S_k(\Gamma_1(N) \cap \Gamma_0(p), \Zbar_p)$ of cusp forms of weight $k$ on $\Gamma_1(N) \cap \Gamma_0(p)$ and coefficients in $\Zbar_p$. The Atkin $U_p$ operator acts on this space with effect on $q$-expansions as $U_p(\sum b_n q^n) = \sum b_{pn} q^n$. The form $f$ gives rise to two eigenforms in $S_k(\Gamma_1(N) \cap \Gamma_0(p), \Zbar_p)$ with corresponding $U_p$-eigenvalues $\lambda,\lambda'$ that are the roots of the polynomial $x^2 - a_p x + p^{k-1}$, cf.\ \cite[Section 4]{gouvea-mazur}.

As usual the number $v_p(a_p)$ is called the $p$-slope of $f$, or simply the slope of $f$ since $p$ is fixed.
\smallskip

By the theorem on local constancy of dimensions of generalized $U_p$ eigenspaces, due to Coleman \cite[Theorem D]{coleman_banach} and, in explicit form (for $p\ge 5$), Wan \cite[Theorem 1.1]{wan}, we can see immediately that the finiteness statement corresponding to Conjecture~\ref{conj:strong_weight_bounds}, but for fixed, finite slope, is true:

\begin{prop}\label{prop:finiteness_for_fixed_slope} Fix $N$, $p$, $m$, and $\alpha \in \Q_{\ge 0}$. There is a constant $k(N,p,m,\alpha)$ depending only on $N$, $p$, $m$, $\alpha$ such that any eigenform on $\Gamma_1(N)$ of $p$-slope $\alpha$ is congruent modulo $p^m$ to an eigenform of the same type and of weight $\le k(N,p,m,\alpha)$.
\end{prop}
\begin{proof} The theorem mentioned above, i.e., local constancy of dimensions of generalized $U_p$ eigenspaces (in $S_k(\Gamma_1(N) \cap \Gamma_0(p), \Zbar_p)$), \cite[Theorem D]{coleman_banach}, \cite[Theorem 1.1]{wan}, implies that the dimension of any such eigenspace is bounded above by a function depending only on the data $N$, $p$, and $\alpha \in \Q_{\ge 0}$. As a consequence, if $f$ is an eigenform on $\Gamma_1(N)$ of fixed $p$-slope $\alpha$ then the field of coefficients $K_{f,p}$ has degree over $\Q_p$ bounded by a function of $N$, $p$, and $\alpha \in \Q_{\ge 0}$ (because the eigenspace in question is stabilized by the Hecke operators and thus the eigenvalues in question arise via diagonalizing Hecke operators in a space of bounded dimension.)

Now, the proof of Proposition~\ref{prop:conj1+2}, specifically the argument provided for item (\ref{item:8}) of that proposition, applied to the set of slope~$\alpha$ eigenforms on~$\Gamma_1(N)$ instead of $\fS(N)$ shows the claim.
\end{proof}

\begin{cor}\label{cor:strong_weight_bounds_bounded_slope} Fixing $N$, $p$, $m$, and $b \in \Q_{\ge 0}$ there is only a finite number of reductions modulo $p^m$ of eigenforms on $\Gamma_1(N)$ of $p$-slope bounded by $b$.
\end{cor}
\begin{proof} This follows from Proposition \ref{prop:finiteness_for_fixed_slope} since the set of $p$-slopes of eigenforms on $\Gamma_1(N)$ is a discrete subset of $\R$, cf.\ \cite{coleman_banach}, remark after Theorem B3.4.
\end{proof}

One can ask for possible explicit values for $k(N,p,m,\alpha)$. In the preprint \cite{mahnkopf1}, J.\ Mahnkopf uses an analysis of the trace formula to embed modular forms of fixed, finite slope into families, albeit not in the rigid analytic sense. His Theorem G of that preprint implies the existence of a constant $k(N,p,m,\alpha)$ as above, and in fact gives a possible explicit value. We will not quote that value as it is a bit involved, but only note that it depends on the dimensions of certain generalized $U_p$ eigenspaces. One is led to the following question.

\begin{question} What is the optimal shape of the above constant $k(N,p,m,\alpha)$?
\end{question}

Additionally, we would like to make the following remark. In connection with this question it is natural first of all to think about utilizing Coleman's theory of $p$-adic families of modular forms. However, it seems that the current state of this theory does not lead to an answer to the above question, at least not in any immediately obvious way. Let us briefly explain this point.

The following statement is a consequence of Coleman's theory (see \cite[Corollary B5.7.1]{coleman_banach}; the proof is only sketched in \cite{coleman_banach}, but see \cite[Section 2]{yamagami} for a detailed proof):

Suppose that $f_0\in S_{k_0}(\Gamma_1(N) \cap \Gamma_0(p), \Zbar_p)$ is a $p$-new eigenform of slope $\alpha$ and with $k_0 > \alpha + 1$. Then there exists $t\in\N$ such that the following holds: whenever $m,k\in \N$ where $k > \alpha + 1$ and
$$
k \equiv k_0 \pmod{p^{m+t}(p-1)}
$$
there exists a $p$-new eigenform $f\in S_k(\Gamma_1(N) \cap \Gamma_0(p), \Zbar_p)$ of slope $\alpha$ such that
$$
f \equiv f_0 \pmod{p^m} .
$$

We can refer to the number $t$ as the ``radius'' of the $p$-adic analytic family passing through $f_0$. The following question naturally arises.

\begin{question}\label{question:coleman_radius} Is the above radius bounded above by a constant depending on $N$, $p$, $\alpha$, but not on $f_0$?
\end{question}

Clearly, an affirmative answer to this question together with an explicit upper bound would instantly lead to an explicit possible value for the above constant $k(N,p,m,\alpha)$.

However, the late Robert Coleman confirmed in an email exchange (August 2013) with the first author that current knowledge about the properties of $p$-adic analytic families of modular forms does not warrant an affirmative answer to Question \ref{question:coleman_radius}.

\subsection{Weak weight bounds modulo \texorpdfstring{$2^m$}{2m} at level 1}\label{subsection:weak weight bounds mod 2^m} As we explained in \ref{subsection:summary,N=1,p=2} above, Theorem \ref{thm:weightbound} was proved before Theorem \ref{thm:weak_weight_bounds} and is of course a very weak version of Theorem \ref{thm:weak_weight_bounds}. We have chosen to retain this theorem though, because of the explicit weight bounds that become possible with this method of proof, see \ref{subsubsection:explicit bounds} below. This raises the obvious question about whether explicit weight bounds in Theorem \ref{thm:weak_weight_bounds} can be obtained via a generalization of this method.

As Theorem \ref{thm:weightbound} is by now superseded by Theorem \ref{thm:weak_weight_bounds}, we shall restrict ourselves to summarizing parts of the arguments leading to Theorem \ref{thm:weightbound}. We refer the interested reader to the preprint version \cite{krw_preprint} of this paper for full details.
\smallskip

For the arguments we need to bring in some standard Eisenstein series on $\SL_2(\Z)$ and discuss the full algebra of modular forms on $\SL_2(\Z)$. However, as above, we will still reserve the word ``eigenform'' for ``normalized cuspidal eigenform for all Hecke operators'' since this is our main focus.

Denote as usual by $Q:=E_4$ and $R:=E_6$ the normalized Eisenstein series on $\SL_2(\Z)$ of weight $4$ and $6$, respectively, and by $\Delta$ Ramanujan's form of weight $12$. We write $M_k(A)$ for the set of modular forms on $\SL_2(\Z)$ of weight $k$ and coefficients in a commutative ring $A$. We remind the reader that we are using the naive definition of ``coefficients in $A$'' throughout this paper, so that we have $M_k(A) = M_k(\Z) \otimes A$. Similarly, $S_k(A)$ will denote cusp forms on $\SL_2(\Z)$ with coefficients in $A$.

Given an even integer $k\ge 4$, one knows, cf.\ for instance \cite[Theorem X.4.3]{lang}, that the forms of shape
$$
\left\{ \begin{array}{ll} Q^a \Delta^c & \mbox{ if } k\equiv 0 \pmod{4} \\
R Q^a \Delta^c & \mbox{ if } k\equiv 2 \pmod{4}\end{array} \right.
$$
with $4a+12c=k$ ($k\equiv 0 \pmod{4}$) and $4a+12c=k-6$ ($k\equiv 2 \pmod{4}$), form a basis for the space of modular forms of weight $k$ with coefficients in $\Z$, i.e., every modular form with $q$-expansion in $\Z[[q]]$ is a $\Z$-linear combination of the above basis forms.

Consequently, if $f\in M_k(\Z /2^m \Z)$ we can write
$$
f = \sum_{4a+12c=k} \alpha_{a,c} Q^a \Delta^c
$$
if $k\equiv 0 \pmod{4}$, and
$$
f = R\cdot \sum_{4a+12c=k-6} \alpha_{a,c} Q^a \Delta^c
$$
if $k\equiv 2 \pmod{4}$, with certain coefficients $\alpha_{a,c} \in \Z /2^m \Z$.

Here, we have abused notation slightly, denoting again by $Q$, $R$, $\Delta$ the images in the appropriate $M_k(\Z /2^m \Z)$ of the forms $Q$, $R$, $\Delta$.

In this expansion of $f$ the coefficients $\alpha_{a,c} \in \Z /2^m \Z$ are uniquely determined. Thus, we can make the following definition.

\begin{dfn} For $f\in S_k(\Z /2^m \Z)$ define the degree $\deg_m{f}$ of $f$ to be the highest power of $\Delta$ occurring in the expansion of $f$ as above.
\end{dfn}

In situations where $m$ does not vary and it is clear what it is, we may suppress the $m$ from the notation and just write $\deg f$ for $\deg_m f$.

For simplicity of notation we have chosen to formulate and prove the next theorem for strong eigenforms that are reductions of eigenforms with coefficients in $\Z_2$. However, the statement (with the same constant $C(m)$) holds more generally for strong eigenforms that are reductions of forms with coefficients in the ring of integers of the maximal unramified extension of $\Q_2$. The proof is, mutatis mutandis, the same as the one that follows below. See also Remark \ref{remark:thm_N=1,p=2} below.

\begin{thm}\label{thm:weightbound} There exists a constant $C(m)$ depending only on $m$ such that the following holds.

Whenever $f\in S_k(\Z/2^m \Z)$ is a strong eigenform modulo $2^m$ that is the reduction of an eigenform with coefficients in $\Z_2$ then
$$
\deg_m{f} \leq C(m).
$$

Any such form is the reduction modulo $2^m$ of a form of weight bounded by a constant $\kappa(m)$ depending only on $m$, that can be taken to be $12C(m)$ for $m=1,2,3$, and to be $6 + 2^{m-2} + 12 C(m)$ if $m\ge 4$.
\end{thm}

The proof of Theorem \ref{thm:weightbound} makes use of the following theorem of Hatada, see \cite[Theorem 1]{hatada1977a} (and \cite{hatada1977b} for further results).

\begin{prop}[Hatada]\label{prop:hatada1} Let $f \in S_k(SL_2(\Z))$ be an eigenform. Then for any prime $p$:
$$
a_p(f) \equiv 0 \pmod{2}.
$$

Consequently, we have $(f \pmod{2}) = (\Delta \pmod{2})$.
\end{prop}

\subsubsection{Serre-Nicolas codes} In this subsection we work exclusively with modular forms mod $2$ on $\SL_2(\Z)$.

As $Q \equiv R \equiv 1 \pmod{2}$, the algebra of modular forms mod $2$ of level $1$ is $\F_2[\Delta]$. We call an element of $\F_2[\Delta]$ {\it even} resp.\ {\it odd} if the occurring powers of $\Delta$ all have even resp.\ odd exponents. By \cite{SNnil}, section 2.2, the subspaces of even and odd elements are both invariant under the action of every Hecke operator $T_{\ell}$ where $\ell$ is an odd prime. If $f\in \F_2[\Delta]$ we can write, in a unique fashion,
$$
f = f_e + f_o
$$
where $f_e$ and $f_o$ are even and odd, respectively.

The main, and in fact decisive, ingredient in the proof of Theorem \ref{thm:weightbound} is the following proposition that can be proved on the basis of Serre--Nicolas' theory, particularly Propositions 4.3 and 4.4 of \cite{SNnil}. We skip the details here and refer the reader to \cite{krw_preprint} for full details.

\begin{prop}\label{prop:oddbound} For every odd integer $k \geq 0$, there exists a constant $N(k)$ depending only on $k$ such that, whenever $f\in \F_2[\Delta]$ is odd with
$$
\sup\{\deg{T_3(f)}, \deg{T_5(f)}\} \leq k,
$$
then:
$$
\deg{f} \leq N(k).
$$
\end{prop}

Before the proof of Theorem \ref{thm:weightbound} we also need the following statement that is an immediate consequence of \cite[Th\'{e}or\`{e}me 1]{serre_p_adic}.

\begin{thm}\label{thm:weight_congruence} Let $f$ and $g$ be modular forms on $\SL_2(\Z)$ with coefficients in $\Z_2$ and weights $k$ and $k_1$, respectively.

Assume that at least one of the coefficients of $f$ is a unit and that we have $f \equiv g \pmod{2^m}$ for some $m\in\N$. Then
$$
k \equiv k_1 \pmod{2^{\alpha(m)}}
$$
where
$$
\alpha(m) = \left\{ \begin{array}{ll} 1 & \mbox{ if } m \le 2\\
m-2 & \mbox{ if } m \ge 3. \end{array} \right.
$$
\end{thm}

Of course, there is a version of the theorem for odd primes, but we will not need that.

\begin{proof} By our definition of modular forms with coefficients in $\Z_2$, there exist modular forms $f' \in M_k(\Z)$ and $g' \in M_{k_1}(\Z)$ such that $f \equiv f' \pmod{2^m}$ and $g \equiv g' \pmod{2^m}$. The hypothesis then says that $f' \equiv g' \pmod{2^m}$. We can now use \cite[Th\'{e}or\`{e}me 1]{serre_p_adic} to finish the proof.
\end{proof}

\begin{proof}[Proof of Theorem \ref{thm:weightbound}] Let us first show that the last statement of the theorem, i.e., the weak weight bound, follows from the first. From the first statement, any strong eigenform modulo $2^m$ is the reduction of a form that can be written as a linear combination of monomials $Q^a \Delta^c$, or $R Q^a \Delta^c$, and where $c\le C(m)$. Now, from the $q$-expansion of $Q$ we have that $Q\equiv 1 \pmod{2^4}$ whence $Q^{2^s} \equiv 1 \pmod{2^{4+s}}$. Suppose that $m\ge 4$. Then for any non-negative $a$ we have $Q^a \equiv Q^{a'} \pmod{2^m}$ for some $a' \le 2^{m-4}$. For such an $a'$ the weight of a monomial $R Q^{a'} \Delta^c$ is $\le 6 + 4\cdot 2^{m-4} + 12C(m) = 6 + 2^{m-2} + 12C(m)$, and the claim follows. For $m=1,2,3$ the claim follows from the congruences $Q\equiv R \equiv 1 \pmod{2^3}$.
\smallskip

We now show the existence of the constant $C(m)$ by induction on $m$. For $m=1$, the result is classical, and it is implied by Proposition \ref{prop:hatada1} that we can take $C(1) = 1$.

Assume $m>1$ and that the statement is true for $m-1$. Let $f\in S_k(\Z /2^m \Z)$ be a strong eigenform modulo $2^m$. The reduction of $f$ modulo $2^{m-1}$ is a strong eigenform modulo $2^{m-1}$. By the induction hypothesis, $\deg_{m-1}{(f \pmod{2^{m-1}})} \le C(m-1)$. Thus, $(f \pmod{2^{m-1}})$ is the reduction modulo $2^{m-1}$ of a form $g$ of weight at most $\kappa(m-1)$ and coefficients in $\Z_2$ and for which the highest power of $\Delta$ occurring in the expansion of $g$ as a sum of monomials $Q^a\Delta^c$, or $RQ^a\Delta^c$, is bounded by $C(m-1)$.

Let the weights of $f$ and $g$ be $k$ and $k_1$, respectively. Since $f$ and $g$ have the same reduction modulo $2^{m-1}$ we know by Theorem \ref{thm:weight_congruence} that
$$
k \equiv k_1 \pmod{2^{\alpha(m-1)}} .
$$

Replacing $f$ by $fQ^{2^s}$ with a sufficiently large $s$, we may assume that $k\ge k_1 + 6$. Write $k = k_1 + t\cdot 2^{\alpha(m-1)}$. Suppose first that $m\ge 5$. Then $Q^{2^{m-5}} \equiv 1 \pmod{2^{m-1}}$ and so the form
$$
g_1 := g\cdot (Q^{2^{m-5}})^t
$$
is of weight $k$, and has the same reduction modulo $2^{m-1}$ as $f$. In the cases $2\le m \le 4$ one also finds a form $g_1$ with these properties, by taking $g_1 := g \cdot Q^{r}$ when $k\equiv k_1 \pmod{4}$, and $g_1 := g \cdot R Q^{r}$ when $k\equiv k_1 + 2 \pmod{4}$ with the appropriate power $r$. It works because $Q \equiv R \equiv 1 \pmod{2^3}$.

Also, the highest power of $\Delta$ occurring when we expand $g_1$ in a sum of monomials in $Q$, $\Delta$, and, possibly, $R$, is bounded from above by $C(m-1)$. This follows because $g$ has that same property. By the argument in the beginning of the proof, it follows that the form $g_1$ is congruent modulo $2^m$ to a form $g_2$ of weight bounded by a constant $w(m)$ depending only on $m$ (specifically, one can take the weight bound from the beginning of the proof with $C(m)$ replaced by $C(m-1)$.) Clearly then, if $\ell$ is any prime number we must have
$$
\deg_m T_\ell g_1 = \deg_m T_\ell g_2 \le \frac{1}{12} w(m).
$$

Consider now that we have
$$
f \equiv g_1 + 2^{m-1} h \pmod{2^m}
$$
with some modular form $h$ with coefficients in $\Z_2$ and weight $k$.

Now, $h \pmod{2}$ is a polynomial in $\Delta$, and if we can bound the degree of this polynomial we are done.

Let $\lambda_2, \lambda_3$ and $\lambda_5$ be respectively the eigenvalues of the operators $T_2$, $T_3$, and $T_5$ associated to $f$. By Proposition \ref{prop:hatada1}, we know that $\lambda_2 \equiv \lambda_3 \equiv \lambda_5 \equiv 0 \pmod{2}$. Thus for $\ell \in \{2,3,5\}$, we have:
$$
T_\ell f \equiv T_\ell g_1 + 2^{m-1} T_\ell h \equiv \lambda_\ell f \equiv \lambda_\ell g_1 \pmod{2^m}
$$
which gives
$$
2^{m-1} T_\ell h \equiv \lambda_\ell g_1 - T_\ell g_1 \pmod{2^m}
$$
for $\ell \in \{2,3,5\}$. Thus,
$$
\deg_m (2^{m-1} T_\ell h) \le \frac{1}{12} w(m)
$$
and hence
$$
\deg_1 (T_\ell h) \le \frac{1}{12} w(m)
$$
for $\ell \in \{2,3,5\}$.

Now split $(h \pmod{2})$ into even and odd parts as explained above:
$$
(h \pmod{2}) = h_e + h_o .
$$
We have
$$
\deg_1{T_\ell h_e} ~,~ \deg_1{T_\ell h_o} \le \frac{1}{12} w(m)
$$
for $\ell \in \{2,3,5\}$.

Consider the classical $U$ and $V$ operators on mod $2$ modular forms. For the even part $h_e$ we have $h_e = \phi^2 = V(\phi)$ for some mod $2$ modular form $\phi$. Since $T_2 \equiv U \pmod{2}$, we see that:
$$
T_2 h_e = UV(\phi) = \phi .
$$

Hence $\deg_1{\phi} \le \frac{1}{12} w(m)$, and so $\deg_1{h_e} \le \frac{1}{6} w(m)$.

For the odd part, we have $\deg_1{T_\ell h_o} \le \frac{1}{12} w(m)$ for $\ell \in \{3,5\}$. By Proposition \ref{prop:oddbound}, it follows that $\deg_1{h_o}$ is bounded by $N(\lfloor \frac{1}{12} w(m) \rfloor)$ if $\lfloor \frac{1}{12} w(m) \rfloor$ is odd, and by $N(\lfloor \frac{1}{12} w(m) \rfloor + 1)$ if $\lfloor \frac{1}{12} w(m) \rfloor$ is even.

We are done.\end{proof}

\begin{rem}\label{remark:thm_N=1,p=2} As we noted above, Theorem \ref{thm:weightbound} holds more generally for strong eigenforms that are reductions of forms with coefficients in the ring of integers of the maximal unramified extension of $\Q_2$, and with the same constant $C(m)$. The proof is essentially the same.

If one allowed non-trivial ramification in the coefficient field the argument breaks down at the induction step: in going from $m-1$ to $m$ one would need not $1$, but $e$ inductional steps, and thus one looses control over the constants involved, or rather, they will depend on $e$.
\end{rem}

\subsubsection{Explicit bounds for low values of \texorpdfstring{$m$}{m}}\label{subsubsection:explicit bounds} It is natural to ask for an explicit ``formula'' for the constants $C(m)$, but we have not been able to find one. For any given $m$, though, a constant $C(m)$ that works in Theorem \ref{thm:weightbound} can in principle be determined, the main obstacle being determining constants $N(\cdot)$ that work in Proposition \ref{prop:oddbound}. We give now examples for the low values $m=1,2,3,4$.

We have $C(1) = 1$ as already remarked and used in the above. To determine constants $C(m)$ for $m=2,3,4$, we refer back to the inequalities appearing at the end of the proof of Theorem \ref{thm:weightbound}:

$$
\deg_1{h_e} \le \frac{1}{6} w(m),
$$
and:
$$
\deg_1{h_o} \leq \left\{ \begin{array}{ll} N(\lfloor \frac{1}{12} w(m) \rfloor) & \mbox{ if $\lfloor \frac{1}{12} w(m) \rfloor$ is odd} \\
N(\lfloor \frac{1}{12} w(m) \rfloor + 1) & \mbox{ if $\lfloor \frac{1}{12} w(m) \rfloor$ is even} \end{array} \right.
$$
where, as in the beginning of the proof, we have:
$$
w(m) \leq \left\{ \begin{array}{ll} 6 + 2^{m-2} + 12C(m-1) & \mbox{ if $m\ge 4$} \\
12C(m-1) & \mbox{ if $m=2,3$.} \end{array} \right.
$$

By the proof of Theorem \ref{thm:weightbound}, it then follows that we can take:
$$
C(m) = \sup\{C(m-1), \lfloor \frac{1}{6} w(m) \rfloor, N(\lfloor \frac{1}{12} w(m) \rfloor) \}.
$$
Using a computer, we compute the following values for the function $N(\cdot)$:
$$
\begin{tabular}{ c | c }
  $k$ & $N(k)$ \\
  \hline
    1 & 5 \\
    5 & 17 \\
    17 & 65 \\
  \hline
\end{tabular}
$$

We also check that the function $N$ is non-decreasing on the set of odd integers $k$ such that $1 \leq k \leq 100$. The calculation of the values of $C(m)$ are summarized in the following table:

$$
\begin{tabular}{ c | c | c | c | c | c}
  $m$ & $w(m) \leq$ & $\lfloor \frac{1}{6} w(m) \rfloor \leq$ & $\lfloor \frac{1}{12} w(m) \rfloor \leq$ & $N(\lfloor \frac{1}{12} w(m) \rfloor) \leq $ & $C(m)$ \\
  \hline
    1 & -   & -   & -  & - & 1  \\
    2 & 19  & 3   & 1  & 5 & 5  \\
    3 & 68  & 11  & 5  & 17 & 17\\
    4 & 214 & 35  & 17 & 65 & 65\\
  \hline
\end{tabular}
$$

A computer search shows that these values are sharp for $m = 2$ and $m = 3$, i.e., in each of these cases there exists a weak eigenform modulo $2^m$ for which $\deg_m$ attains the upper bound $C(m)$. We do not know whether the value for $C(4)$ is sharp, as the calculations become too demanding.

\subsection{Some numerical data}\label{subsection:numerical_data} We will finally present a bit of numerical data that can be seen as an experimental approach to the constant $k(N,p,m)$ of Theorem~\ref{thm:weak_weight_bounds}.

The following table summarizes our data. The explanation of the table is this: for each entry we generated all eigenforms of weight $\le k_{max}$ on the group in question; then, we looked at the reduction modulo $p^m$ of each of these eigenforms $f$ and determined the smallest weight $k(f)$ where it occurs weakly modulo $p^m$; the number $k$ in the corresponding entry is the maximum of the $k(f)$ for $f$ in this particular set of eigenforms.

\begin{tabular}{ c | c | c | c | c } Group & $p$ & $m$ & $k$ & $k_{max}$ \\
\hline
$\Gamma_0(1)$ & 5 & 2 & 76 & 320 \\
\hline
$\Gamma_0(1)$ & 5 & 3 & 276 & 288 \\
\hline
$\Gamma_0(1)$ & 7 & 2 & 148 & 246 \\
\hline
$\Gamma_0(1)$ & 11 & 2 & 364 & 374 \\
\hline
$\Gamma_0(2)$ & 5 & 2 & 76 & 174 \\
\hline
$\Gamma_0(2)$ & 5 & 3 & 276 & 316 \\
\hline
$\Gamma_0(2)$ & 7 & 2 & 148 & 246 \\
\hline
$\Gamma_0(2)$ & 11 & 2 & 364 & 370 \\
\hline
$\Gamma_0(3)$ & 5 & 2 & 76 & 174 \\
\hline
$\Gamma_0(3)$ & 5 & 3 & 276 & 278 \\
\hline
$\Gamma_0(3)$ & 7 & 2 & 148 & 222 \\
\hline
$\Gamma_0(5)$ & 5 & 2 & 76 & 138 \\
\hline
$\Gamma_0(9)$ & 5 & 2 & 76 & 150 \\
\hline
$\Gamma_1(3)$ & 5 & 2 & 76 & 174 \\
\hline
$\Gamma_1(3)$ & 5 & 3 & 276 & 296 \\
\hline
$\Gamma_1(3)$ & 7 & 2 & 148 & 204 \\
\hline
$\Gamma_1(11)$ & 5 & 2 & 76 & 88 \\
\hline
\end{tabular}
\smallskip

Thus, the number $\kappa$ can be seen as an ``experimental value'' for the constants occurring in Theorem~\ref{thm:weak_weight_bounds}. The values of $k$ in the table would be consistent with a more precise version of the statement of Theorem~\ref{thm:weak_weight_bounds}, namely that it holds with a constant $k(N,p,m)$ that is in fact independent of $N$, and has the following precise value:
$$
k(N,p,m) = 2p^m + p^2 + 1
$$
when $m\ge 2$.

However, we certainly do not feel that the extent of our data above warrants us actually making this conjecture. It merely raises the question of whether this is a general bound. If it is, we feel that it would be suggestive of a relatively ``elementary'' reason for that bound. We hope to return to this question elsewhere.

The reader may recall that there is in fact an established value for $k(N,p,1)$, namely $k(N,p,1) = p^2+p$, cf.\ the work of Jochnowitz in \cite{jochnowitz_local_components}, \cite{jochnowitz_finiteness}, specifically \cite[Lemma 4.4]{jochnowitz_local_components}. It might then be objected that this value does not seem to fall into any easily discernible pattern with the above values for $m\ge 2$. However, one should remember that the established value $k(N,p,1) = p^2+p$ is intimately connected with the behavior of the $\theta$ operator modulo $p$ and that, as the paper \cite{ck} shows, the $\theta$ operator modulo $p^m$ behaves in a much more complicated way when $m\ge 2$. Hence the authors do not feel that one has very much guidance from the established value of $k(N,p,1)$ when it comes to guessing the optimal shape of the constant $k(N,p,m)$ for $m\ge 2$.

The computations were done in MAGMA, cf.\ \cite{Magma}.

\subsection{Further questions} All of the results above concerning ``weight bounds'' are weight bounds for strong eigenforms. Thus, a natural question to ask is whether there exist ``weak weight bounds for weak eigenforms for fixed $N$, $p$, and $m$ as above''. Cf.\ the discussion in \ref{subsection:weight_bounds}.

We do not know the answer to this question, but we would like to remark that the following statement can be proved and can be seen as pointing somewhat in the direction of an affirmative answer.

\begin{prop} There exists a constant $w=w(N,p,m,e)$ such that: if $f\in S_k(O/\p^{\epsilon (m-1)+1})$ is a weak eigenform mod $p^m$ where $O$ is the valuation ring of a finite extension $K/\Q_p$ with ramification index $\epsilon \le e$ and $\p$ the maximal ideal of $O$, then there exists a weak eigenform $g$ in weight $\le w$ such that:
$$
f \equiv g \pmod{p^m} .
$$
\end{prop}

Thus, if condition $\mathbf{(B)}$ of \ref{subsection:finiteness_conditions} holds, the above proposition implies the existence of ``weak weight bounds for weak eigenforms''. The question of establishing such bounds unconditionally is naturally raised.

\section{Acknowledgements} The authors are very grateful to Frank Calegari for having explained to them the idea of proof of Theorem~\ref{thm:weak_weight_bounds} and for letting them include it in the paper.
The authors would also like to thank Gebhard B\"ockle for various interesting conversations about the questions in this paper.
G.W. thanks Panagiotis Tsaknias for a huge number of very enlightening and insightful remarks on this
and related subjects.
I.K.\ and N.R.\ acknowledge support by grants from the Danish Council for Independent Research, and from VILLUM FONDEN through the network for Experimental Mathematics in Number Theory, Operator Algebras, and Topology.
G.W.\ acknowledges support by the Fonds National de la Recherche Luxembourg (INTER/DFG/FNR/12/10/COMFGREP).

Additionally, the authors would like to thank the anonymous referee for a careful reading and suggestions that helped improve the presentation.

\end{document}